\documentclass{amsart}

\usepackage{amsmath}
\usepackage{amsthm}
\usepackage{amssymb}

\usepackage{comment}

\usepackage{hyperref}

\newtheorem{theorem}{Theorem}[section]
\newtheorem{definition}[theorem]{Definition}
\newtheorem{lemma}[theorem]{Lemma}
\newtheorem{corollary}[theorem]{Corollary}
\newtheorem{remark}[theorem]{Remark}

\newcommand{\R}{{\mathbb R}}
\newcommand{\N}{{\mathbb N}}
\newcommand{\Z}{{\mathbb Z}}
\newcommand{\C}{{\mathbb C}}

\newcommand{\nn}{\nonumber}
\newcommand{\be}{\begin{equation}}
\newcommand{\ee}{\end{equation}}

\newcommand{\ol}{\overline}
\newcommand{\ti}{\tilde}

\newcommand{\E}{\mathrm{e}}

\newcommand{\sgn}{\mathrm{sgn}}
\newcommand{\tr}{\mathrm{tr}}

\DeclareMathOperator{\Ran}{Ran}

\newcommand{\floor}[1]{\lfloor#1 \rfloor}
\newcommand{\ceil}[1]{\lceil#1 \rceil}

\DeclareMathOperator{\arccot}{arccot}

\newcommand{\eps}{\varepsilon}

\newcommand{\sig}{\sigma}
\newcommand{\lam}{\lambda}

\numberwithin{equation}{section}

\begin{document}

\title[Effective Pr\"ufer Angles]{Effective Pr\"ufer Angles and
Relative Oscillation Criteria}

\author[H. Kr\"uger]{Helge Kr\"uger}
\address{Department of Mathematics\\ Rice University\\ Houston\\ TX 77005\\ USA}
\email{\href{mailto:helge.krueger@rice.edu}{helge.krueger@rice.edu}}
\urladdr{\href{http://math.rice.edu/~hk7/}{http://math.rice.edu/\~{}hk7/}}

\author[G. Teschl]{Gerald Teschl}
\address{Faculty of Mathematics\\
Nordbergstrasse 15\\ 1090 Wien\\ Austria\\ and International Erwin Schr\"odinger
Institute for Mathematical Physics, Boltzmanngasse 9\\ 1090 Wien\\ Austria}
\email{\href{mailto:Gerald.Teschl@univie.ac.at}{Gerald.Teschl@univie.ac.at}}
\urladdr{\href{http://www.mat.univie.ac.at/~gerald/}{http://www.mat.univie.ac.at/\~{}gerald/}}

\thanks{{\it Research supported by the Austrian Science Fund (FWF) under Grant No.\ Y330}}


\begin{abstract}
We present a streamlined approach to relative oscillation criteria based
on effective Pr\"ufer angles adapted to the use at the edges of the essential
spectrum.

Based on this we provided a new scale of oscillation criteria for general
Sturm--Liouville operators which answer the question whether a perturbation
inserts a finite or an infinite number of eigenvalues into an essential
spectral gap. As a special case we recover and generalize the Gesztesy--\"Unal criterion
(which works below the spectrum and contains classical criteria by Kneser, Hartman, Hille, and
Weber) and the well-known results by Rofe-Beketov including the extensions by Schmidt.
\end{abstract}

\maketitle

\section{Introduction}

In this article we want to use relative oscillation theory
and apply it to obtain criteria for when an edge of an essential spectral gap
is an accumulation point of eigenvalues for Sturm--Liouville operators
\be
\tau = \left(-\frac{d}{dx} p \frac{d}{dx} + q\right),
\quad  \text{on}\quad (a,b).
\ee
Without loss of generality we will assume that $a\in\R$ is a regular endpoint
and that $b$ is limit point. Furthermore, we always assume the usual
local integrability assumptions on the coefficients (see Section~\ref{sec:res}).

We will assume that $H_0$ is a given background operator associated with
$\tau_0 = (-\frac{d}{dx} p_0 \frac{d}{dx} + q_0)$ (think e.g.\ of a periodic
operator) and that $E$ is a boundary point of the essential spectrum of $H_0$
(which is not an accumulation point of eigenvalues). Then we want to know
when a perturbation $\tau_1 = (-\frac{d}{dx} p_1 \frac{d}{dx} + q_1)$
gives rise to an infinite number of eigenvalues accumulating at $E$.
By relative oscillation theory, this question reduces to the question of when
a given operator $\tau_1-E$ is relatively oscillatory with respect to $\tau_0-E$
(cf.\ Section~\ref{sec:relosc}).

In the simplest case $\tau_0 = - \frac{d^2}{dx^2}$, $E=0$, Kneser \cite{kn} showed that
the borderline case is given by ($p_1=p_0=1$)
\be
q_1(x) = \frac{\mu}{x^2},
\ee
where the critical constant is given by $\mu_c=-\frac{1}{4}$. That is,
for $\mu<\mu_c$ the perturbation is oscillatory and for $\mu>\mu_c$ it
is nonoscillatory. In fact, later on Hartman \cite{hrt1}, Hille \cite{hl1}, and Weber \cite{wbr}
gave a whole scale of criteria addressing the case $\mu=\mu_c$. Recently this result was further
generalized by Gesztesy and \"Unal \cite{gu}, who showed that for Sturm--Liouville
operators (with $p_1=p_0$) the borderline case for $\tau_0-E$, $E=\inf \sig(H_0)$, is given by
\be \label{guebc}
q_1(x) = q_0(x) + \frac{\mu}{p_0(x) u_0(x)^2 v_0(x)^2},
\ee
where the critical constant is again $\mu_c=-\frac{1}{4}$. Here
$u_0$ is a minimal (also principal) positive solution of $\tau_0 u =0$ and $v_0$ is a second
linearly independent solution with Wronskian $W(u_0,v_0)=1$.
Since for $p_0=1$, $q_0=0$ we have $u_0=1$ and $v_0=x$, this result contains
Kneser's result as a special case. Moreover, they also provided a scale
of criteria for the case $\mu=\mu_c$.

While Kneser's result is classical, the analogous question for a periodic background $q_0$
(and $p_0=1$) was answered much later by Rofe-Beketov in a series of papers
\cite{rb1}--\cite{rb5} in which he eventually showed that the borderline case is again given by
\be\label{rbbc}
q_1(x) = q_0(x) + \frac{\mu}{x^2},
\ee
where the critical constant $\mu_c$ can be expressed in terms of the Floquet
discriminant. His result was recently extended by Schmidt \cite{kms} to the case
$p_0=p_1 \ne 1$ and Schmidt also provided the second term in the
case $\mu=\mu_c$.

These results raised the question for us, if there is a generalization of the
Gesztesy--\"Unal result which holds inside any essential spectral gap (and not just the lowest).
Clearly (\ref{guebc}) makes no sense, since above the lowest edge of the essential spectrum,
all solutions of $\tau_0 u = E u$ have an infinite number of zeros.
However, in the periodic background case, as in the constant background case, there
is one solution $u_0$ which is bounded and a second solution $v_0$  which grows like $x$.
Hence, at least formally, the Gesztesy--\"Unal result explains why the borderline
case is given by (\ref{rbbc}). However, their proof has positivity of $H_0-E$ as
the main ingredient and thus cannot be generalized to the case above the
infimum of the spectrum.

In summary, there are two natural open problems which we want to address in this
paper: First of all, the whole scale of oscillation criteria inside essential spectral gaps for
critically perturbed periodic operators. Secondly, what is the analog of the
Gesztesy--\"Unal result (\ref{guebc}) inside essential spectral gaps? Based
on the original ideas of Rofe-Beketov and the extensions by Schmidt, we
will provide a streamlined approach to the subject which will recover and at the
same time extend all previously mentioned results. For example, we will derive an averaged
version of the Gesztesy--\"Unal result (including the whole scale) which, to the best
of our knowledge, is new even in the case originally considered by Kneser.

Concerning the Gesztesy--\"Unal result we show that if $u_0$, $v_0$ are two linearly
independent solutions of $\tau_0 u = E u$ with Wronskian $W(u_0,v_0)=1$ such that
there are functions $\alpha(x)>0$ and $\beta(x) \lessgtr 0$ satisfying
$u_0(x)=O(\alpha(x))$ and $v_0(x) - \beta(x) u_0(x)=O(\alpha(x))$ as $x\to\infty$.
Then $(p_0=p_1)$
\be
q_1(x) = q_0(x) + \frac{\mu\, \beta'(x)}{\alpha(x)^2 \beta(x)^2},
\ee
is relatively oscillatory if $\limsup_{x\to\infty} \frac{\mu}{\ell} \int_x^{x+\ell} u_0(t)^2 \alpha(t)^{-2} dt
< -\frac{1}{4}$ and relatively non\-oscillatory if $\liminf_{x\to\infty} \frac{\mu}{\ell}
\int_x^{x+\ell} u_0(t)^2 \alpha(t)^{-2} dt > -\frac{1}{4}$.
By virtue of d'Alembert's formula, this reduces to (\ref{guebc}) for $E$ at the bottom of the spectrum,
where we can set $\alpha=u_0$ and $\beta=\frac{v_0}{u_0}=\int p_0^{-1} u_0^{-2}$.

We will also be able to include the case $p_0\ne p_1$ with no additional effort and
we will provide a full scale of criteria in all cases.

\section{Main results}
\label{sec:res}

In this section we will summarize our main results. We will go from the simplest to
the most general case rather than the other way round for two reasons: First of all,
in our proofs, which will be given in Section~\ref{sec:epa}, we will also advance in this
direction and show how the general case follows from the special one. In particular,
this approach will allow for much simpler proofs. Secondly,
several of the special cases can be proven under somewhat weaker assumptions.

We will consider Sturm--Liouville operators on $L^{2}((a,b), r\,dx)$
with $-\infty \leq a<b \le \infty$ of the form
\begin{equation} \label{stli}
\tau = \frac{1}{r} \Big(- \frac{d}{dx} p \frac{d}{dx} + q \Big),
\end{equation}
where the coefficients $p,q,r$ are real-valued
satisfying
\begin{equation}
p^{-1},q,r \in L^1_{loc}(a,b), \quad p,r>0.
\end{equation}
We will use $\tau$ to describe the formal differentiation expression and
$H$ for the operator given by $\tau$ with separated boundary conditions at
$a$ and/or $b$.

If $a$ (resp.\ $b$) is finite and $q,p^{-1},r$ are in addition integrable
near $a$ (resp.\ $b$), we will say $a$ (resp.\ $b$) is a \textit{regular}
endpoint.

Our objective is to compare two Sturm--Liouville operators $\tau_0$ and $\tau_1$
given by
\be \label{stlij}
\tau_j = \frac{1}{r} \Big(- \frac{d}{dx} p_j \frac{d}{dx} + q_j \Big), \qquad j=0,1.
\ee
Throughout this paper we will abbreviate
\be
\Delta p = \frac{1}{p_0} - \frac{1}{p_1} = \frac{p_1-p_0}{p_1 p_0}, \qquad
\Delta q = q_1 -q_0.
\ee
Moreover, without loss of generality we will assume that for both operators $a\in\R$
is a regular endpoint and that $b$ is limit point (i.e., $(\tau - z) u$ has at
most one $L^2$ solution near $b$). 

We begin with the case where $E$ is the infimum of the spectrum of $H_0$.
Suppose that $(\tau_0-E)u=0$ has a positive solution and let $u_0$ be the
corresponding minimal (principal) positive solution of $(\tau_0-E) u_0= 0$ near $b$, that is,
$$
\int^b \frac{dt}{p_0(t)u_0(t)^2} = \infty.
$$
By d'Alembert's formula there is a second linearly independent solution
\be\label{eq:dAl}
v_0(x) = u_0(x) \int_a^x \frac{dt}{p_0(t) u_0(t)^2}
\ee
satisfying $W(u_0,v_0)=1$.

Recall that $\tau_1-E$ is called nonoscillatory if one solutions of
$(\tau_1-E)u$ has a finite number of zeros in $(a,b)$. By Sturm's
comparison theorem, this is then the case for all (nontrivial) solutions.

\begin{theorem}\label{thm:gu}
Suppose $\tau_0-E$ has a positive solution and  let $u_0$ be a minimal positive solution.
Define $v_0$ by d'Alembert's formula (\ref{eq:dAl}) and suppose
\be\label{cond:gu}
\lim_{x\to b} p_0 v_0\, p_0 u_0' \Delta p = \lim_{x\to b} p_0 \Delta p = 0.
\ee
Then $\tau_1-E$ is oscillatory if
\be
\limsup_{x\to b} p_0 v_0^2 ( u_0^2 \Delta q + (p_0 u_0')^2 \Delta p) < -\frac{1}{4}
\ee
and nonoscillatory if
\be
\liminf_{x\to b} p_0 v_0^2 ( u_0^2 \Delta q + (p_0 u_0')^2 \Delta p) > -\frac{1}{4}.
\ee
\end{theorem}

\begin{remark}
(i). If $u_0$ is a positive solution which is not minimal near $b$,
that is $\int^b p_0(t)^{-1} u_0(t)^{-2} dt<\infty$, then
$$
v_0(x) = u_0(x) \int_x^b \frac{dt}{p_0(t) u_0(t)^2}
$$
is a minimal positive solution.

(ii). Clearly, the requirement that $\tau_0-E$ has a positive solution can be weakened to 
$\tau_0-E$ being nonoscillatory. In fact, after increasing $a$ beyond the last zero
of some solution, we can reduce the nonoscillatory case to the positive one.

(iii). Note that the coefficient $r$ does not enter since we have chosen it to be the
same for $\tau_0$ and $\tau_1$.
\end{remark}

\noindent
The special case $\Delta p=0$ is the Gesztesy--\"Unal oscillation criterion \cite{gu}. It is not
hard to see (cf.\ Section~\ref{sec:floq}), that it can be used to give a simple proof of
Rofe-Beketov's result at the infimum of the essential spectrum (another simple proof for
this case was given by Schmidt in \cite{kms1}, which also contains nice
applications to the spectrum of radially periodic Schr\"odinger operators in the plane).
Moreover, it is only the first one in a whole scale of oscillation criteria. To get the remaining ones,
we start by demonstrating that Kneser's classical result together with all its generalizations
follows as a special case.

To see this, we recall the iterated logarithm $\log_n(x)$ which is defined recursively via
$$
\log_0(x) = x, \qquad \log_n(x) = \log(\log_{n-1}(x)).
$$
Here we use the convention $\log(x)=\log|x|$ for negative values of $x$. Then
$\log_n(x)$ will be continuous for $x>\E_{n-1}$ and positive for $x>\E_n$, where
$\E_{-1}=-\infty$ and $\E_n=\E^{\E_{n-1}}$. Abbreviate further
$$
L_n(x) = \frac{1}{\log_{n+1}'(x)} = \prod_{j=0}^n \log_j(x),\qquad
Q_n(x) = -\frac{1}{4} \sum_{j=0}^{n-1} \frac{1}{L_j(x)^2}.
$$
Here and in what follows the usual convention that $\sum_{j=0}^{-1} \equiv 0$ is used,
that  is, $Q_0(x)=0$.

\begin{corollary} \label{cor:khwh}
Fix some $n\in\N_0$ and $(a,b)=(\E_n,\infty)$. Let
$$
p_0(x) = 1, \qquad
q_0(x) = Q_n(x).
$$
and suppose
\be\label{eq:khwh}
p_1(x) = 1 +o\Big(\frac{x}{L_n(x)}\Big).
\ee
Then $\tau_1$ is oscillatory if
\be
\limsup_{x\to\infty} L_n(x)^2 \left( \Delta q(x) + \frac{\delta_n}{4 x^2} \Delta p(x) \right) < -\frac{1}{4}
\ee
and nonoscillatory if
\be
\liminf_{x\to\infty} L_n(x)^2 \left( \Delta q(x) + \frac{\delta_n}{4 x^2} \Delta p(x) \right) > -\frac{1}{4},
\ee
where $\delta_n = 0$ for $n=0$ and $\delta_n =1$ for  $n \ge 1$.
\end{corollary}

\begin{proof}
Observe
$$
u_0(x)= \sqrt{L_{n-1}(x)},\qquad
v_0(x) = u_0(x) \log_n(x) = \sqrt{\log_n(x) L_n(x)}
$$
(where we set $L_{-1}(x)=1$) and check
$$
q_0 = \frac{u_0''}{u_0} =
\frac{1}{4}\bigg(\frac{L_n'}{L_n}\bigg)^{\!\!2} + \frac{1}{2} \left(\frac{L_n'}{L_n}\right)'
= \frac{1}{4}\bigg(\sum_{j=1}^n \frac{1}{L_j}\bigg)^{\!\!2} -
\frac{1}{2} \sum_{j=1}^n \frac{1}{L_j} \sum_{k=1}^j \frac{1}{L_k} = Q_n
$$
using $L_n' = L_n \sum_{j=1}^n L_j^{-1}$.
Then
$$
p_0 v_0^2 ( u_0^2 \Delta q + (p_0 u_0')^2 \Delta p) = L_n(x)^2
\left( \Delta q(x) + \frac{1}{4} \Big(\sum_{j=0}^{n-1} \frac{1}{L_j(x)}\Big)^2
\Delta p(x) \right)
$$
where $\sum_{j=0}^{n-1} \frac{1}{L_j(x)} = 0$ for $n=0$ and
$\sum_{j=0}^{n-1} \frac{1}{L_j(x)} = x^{-1}+ o(x^{-1})$ for  $n \ge 1$.
\end{proof}

\noindent
The special case $n=0$ and $\Delta p=0$ is Kneser's classical result \cite{kn}.
The extension to $n\in\N_0$ and $\Delta p=0$ is due to Weber \cite{wbr}, p.53--62, and
was later rediscovered by Hartman \cite{hrt1} and Hille \cite{hl1}.

In fact, there is an analogous scale of oscillation criteria which contains
Theorem~\ref{thm:gu} as the first one $n=0$:

\begin{theorem}\label{thm:gun}
Fix $n\in\N_0$.
Suppose $\tau_0-E$ has a positive solution and  let $u_0$ be a minimal positive solution.
Define $v_0$ by d'Alembert's formula (\ref{eq:dAl}) and suppose
$$
p_0 v_0\, p_0 u_0'  \Delta p = o\big(\frac{(v_0/u_0)^2}{L_n(v_0/u_0)^2}\big), \qquad
p_0 \Delta p = o\big(\frac{(v_0/u_0)^2}{L_n(v_0/u_0)^2}\big).
$$
Then $\tau_1-E$ is oscillatory if
\be
\limsup_{x\to b} L_n(\frac{v_0}{u_0})^2 \left( p_0 u_0^2 ( u_0^2 \Delta q + (p_0 u_0')^2 \Delta p)
-Q_n(\frac{v_0}{u_0}) \right) < -\frac{1}{4}
\ee
and nonoscillatory if
\be
\liminf_{x\to b} L_n(\frac{v_0}{u_0})^2 \left( p_0 u_0^2 ( u_0^2 \Delta q + (p_0 u_0')^2 \Delta p)
-Q_n(\frac{v_0}{u_0}) \right) > -\frac{1}{4}.
\ee
\end{theorem}

\noindent
The special case $\Delta p=0$ is again due to \cite{gu}. The special case $\tau_0=-\frac{d^2}{dx^2}$
gives again Corollary~\ref{cor:khwh}, however, under the (for $n>0$) somewhat stronger condition
$\lim_{x\to\infty} x^{-2} L_n(x)^2 \Delta p(x) =0$.

Moreover, there is even a version which takes averaged (rather than pointwise) deviations from
the borderline case:

\begin{theorem}\label{thm:gua}
Suppose $\tau_0-E$ has a positive solution on $(a,\infty)$ and let $u_0$ be a minimal positive solution.
Define $v_0$ by d'Alembert's formula (\ref{eq:dAl}) and suppose
$$
p_0 v_0^2 (u_0^2 \Delta q + (p_0 u_0')^2 \Delta p) =O(1), \qquad
\lim_{x\to\infty} p_0 v_0\, p_0 u_0' \Delta p = \lim_{x\to\infty} p_0 \Delta p = 0,
$$
and $\rho= (p_0 u_0 v_0)^{-1}$ satisfies $\rho=o(1)$ and
$\frac{1}{\ell} \int_0^\ell \left|\rho(x+t) -\rho(x) \right| dt = o(\rho(x))$.

Then $\tau_1-E$ is oscillatory if
\be
\limsup_{x\to\infty} \frac{1}{\ell} \int_x^{x+\ell} p_0(t) v_0^2(t)
\left( u_0(t)^2 \Delta q(t) + (p_0(t) u_0'(t))^2 \Delta p(t)\right) dt < -\frac{1}{4}
\ee
and nonoscillatory if
\be
\liminf_{x\to\infty} \frac{1}{\ell} \int_x^{x+\ell} p_0(t) v_0^2(t)
\left( u_0(t)^2 \Delta q(t) + (p_0(t) u_0'(t))^2 \Delta p(t)\right) dt > -\frac{1}{4}.
\ee
\end{theorem}

\noindent
Again we have

\begin{corollary}
Fix some $n\in\N_0$ and $(a,b)=(\E_n,\infty)$. Let
$$
p_0(x) = 1, \qquad
q_0(x) = Q_n(x)
$$
and suppose
$$
q_1(x) = Q_n(x) + O\Big(\frac{1}{L_n(x)^2}\Big), \qquad
p_1(x) = 1 + \begin{cases} o(1), & n=0,\\
O\Big(\frac{x^2}{L_n(x)^2}\Big), & n\ge 1.
\end{cases}
$$
Then $\tau_1$ is oscillatory if
\be
\inf_{\ell>0} \limsup_{x\to\infty} \frac{1}{\ell} \int_x^{x+\ell} L_n(t)^2 \left( \Delta q(t) +
\frac{\delta_n}{4 t^2} \Delta p(t) \right) dt < -\frac{1}{4}
\ee
and nonoscillatory if
\be
\sup_{\ell>0} \liminf_{x\to\infty} \frac{1}{\ell} \int_x^{x+\ell} L_n(t)^2 \left( \Delta q(t) +
\frac{\delta_n}{4 t^2} \Delta p(t) \right) dt > -\frac{1}{4},
\ee
where $\delta_n = 0$ for $n=0$ and $\delta_n =1$ for  $n \ge 1$.
\end{corollary}

\noindent
To the best of our knowledge this result is new even in the special case $n=0$,
in which we have that $\tau_1$ with $q_1= O(x^{-2})$ and $p_1 = 1 + o(1)$
is oscillatory if
\be
\inf_{\ell>0} \limsup_{x\to\infty} \frac{1}{\ell} \int_x^{x+\ell} t^2 q_1(t) dt < -\frac{1}{4}
\ee
and nonoscillatory if
\be
\sup_{\ell>0} \liminf_{x\to\infty} \frac{1}{\ell} \int_x^{x+\ell} t^2 q_1(t) dt > -\frac{1}{4}.
\ee
There is also a scale of criteria given in Theorem~\ref{thm:guan} which contains
Theorem~\ref{thm:gua} as the special case $n=0$. Note that the criterion is similar in
spirit to the Hille--Wintner criterion (see e.g., \cite{sw}) which states that $\tau_1$, with
$q_1$ integrable, is oscillatory if
\be
\limsup_{x\to\infty} x \int_x^\infty q_1(t) dt < -\frac{1}{4}
\ee
and nonoscillatory if
\be
\liminf_{x\to\infty} x \int_x^\infty q_1(t) dt > -\frac{1}{4}.
\ee
Result similar in spirit which are applicable at the bottom of the essential spectrum of
periodic operators were given by Khrabustovskii \cite{kh1}, \cite{kh2}.

Our next aim is to extend these result to the case where we are not necessarily at the infimum of the
spectrum of $H_0$. We will again assume that there is a {\em minimal} solution $u_0$ (i.e., one
solution with minimal growth) such that all other solutions are of the form
$v_0= \ti{v}_0 + \beta u_0$, where $\ti{v}_0$ grows like $u_0$ and $\beta$ is some positive or negative
function, which measures how much faster $v_0$ grows on average with respect to $u_0$.
For example, in the case of periodic operators we will have that $u_0$ (and hence $\ti{v}_0$)
is bounded and $\beta(x)= \pm x$ (the sign depending on whether we are at a lower or upper edge of
the spectral band). Moreover, since expressions like $\liminf p_0 u_0^2 v_0^2 \Delta q$
will just be zero if $u_0$ (and $v_0$) have zeros, we will average over some interval.
To avoid problems at finite end points we will choose $b=\infty$ from now on.

But first of all we will state our growth condition more precisely:

\begin{definition}\label{def:regedge}
A boundary point $E$ of the essential spectrum of $H_0$ will be called admissible if there
is a minimal solution $u_0$ of $(\tau_0-E) u_0= 0$ and a second linearly independent
solution $v_0$ with $W(u_0,v_0)=1$ such that
$$
\begin{pmatrix} u_0 \\ p_0 u_0' \end{pmatrix} = O(\alpha), \quad
\begin{pmatrix} v_0 \\ p_0 v_0' \end{pmatrix} - \beta\begin{pmatrix} u_0 \\ p_0 u_0' \end{pmatrix} =
o(\alpha \beta)
$$
for some weight functions $\alpha>0$, $\beta\lessgtr 0$, where $\beta$ is absolutely continuous such
that $\rho=\frac{\beta'}{\beta}>0$ satisfies $\rho(x)=o(1)$ and 
$\frac{1}{\ell} \int_0^\ell \left|\rho(x+t) -\rho(x) \right| dt = o(\rho(x))$.
\end{definition}

\noindent
Clearly, two solutions as in Definition~\ref{def:regedge} can always be found if one chooses
$\alpha$ to grow faster than any solution. However, such a choice will only produce nonoscillatory
perturbations! Hence, in order to get finite critical coupling constants below, the {\em right} choice
for $\alpha$ and $\beta$ will be crucial. Roughly speaking $\alpha$ needs to chosen such that
$\frac{1}{\ell} \int_x^{x+\ell} \frac{u_0(t)^2}{\alpha(t)^2} dt$
remains bounded from above and below by some positive constants as $x\to\infty$. Moreover,
it turns out that the sign of $\beta$ will depend on whether $E$ is a lower or upper boundary of the essential spectrum (i.e., if the essential spectral gap starts below or above $E$). This is related
to our requirement $W(u_0,v_0)=1$.

Note that a second linearly independent solution $v_0$ with $W(u_0,v_0)=1$ can be
obtained by Rofe-Beketov's formula
\begin{align*}
v_0(x) =&
u_0(x) \int^x \frac{(q_0(t)+p_0(t)^{-1} - E\, r(t))(u_0(t)^2 - (p_0(t) u_0'(t))^2)}{(u_0(t)^2 + (p_0(t) u_0'(t))^2)^2} dt\\
& - \frac{p_0(x) u_0'(x)}{u_0(x)^2 + (p_0(x) u_0'(x))^2}
\end{align*}
(the case $p_0\ne 1$ is due to \cite{kms}). In fact, this formula can be used to show that
these assumptions are satisfied for certain almost periodic potentials (see \cite[Sect.~6.4]{rbk}).

In this case we will need to look at the difference between the zeros of two solutions
$u_j$, $j=0,1$, of $(\tau_j-E) u_j =0$. We will call  $\tau_1-E$ is relatively nonoscillatory with
respect to $\tau_0-E$ if the difference between the number of zeros of $u_1$ and $u_0$
when restricted to $(a,c)$ remains bounded as $c\to \infty$, and relatively
oscillatory otherwise. Further details and the
connection with the spectra will be given in Section~\ref{sec:relosc}.

Now, we come to our main result.

\begin{theorem}\label{thm:main}
Suppose $E$ is an admissible boundary point of the essential spectrum of $\tau_0$,
with $u_0$, $v_0$ and $\alpha$, $\beta$ as in Definition~\ref{def:regedge}.
Furthermore, suppose that we have
\be
\Delta q, \Delta p = O\big(\frac{\beta'}{\alpha^2 \beta^2}\big).
\ee

Then $\tau_1-E$ is relatively oscillatory with respect to $\tau_0-E$ if
\be
\inf_{\ell>0} \limsup_{x\to\infty} \frac{1}{\ell} \int_x^{x+\ell} \frac{\beta(t)^2}{\beta'(t)}
\left( u_0(t)^2 \Delta q(t) + (p_0(t) u_0'(t))^2 \Delta p(t)\right) dt < -\frac{1}{4}
\ee
and relatively nonoscillatory with respect to $\tau_0-E$ if
\be
\sup_{\ell>0} \liminf_{x\to\infty} \frac{1}{\ell} \int_x^{x+\ell} \frac{\beta(t)^2}{\beta'(t)}
\left( u_0(t)^2 \Delta q(t) + (p_0(t) u_0'(t))^2 \Delta p(t)\right) dt > -\frac{1}{4}.
\ee
\end{theorem}

\noindent
We remark that the growth conditions from Definition~\ref{def:regedge} on the derivatives $p_0 u_0'$
and $p_0 v_0'$ are not needed if $\Delta p=0$. Similarly, the growth conditions on $u_0$ and $v_0$
are not needed if $\Delta q=0$. 

In the case where $\Delta q$ and $\Delta p$ have precise asymptotics we have:

\begin{corollary}
Suppose
\be
\Delta q = \mu \frac{\beta'}{\alpha^2 \beta^2} (1 + o(1)), \qquad
\Delta p = \nu \frac{\beta'}{\alpha^2 \beta^2} (1 + o(1)).
\ee

Then $\tau_1-E$ is relatively oscillatory with respect to $\tau_0-E$ if
\be\label{eq:reloscs}
\inf_{\ell>0} \limsup_{x\to\infty} \frac{1}{\ell} \int_x^{x+\ell}
\left( \mu \frac{u_0(t)^2}{\alpha(t)^2} + \nu \frac{(p_0(t) u_0'(t))^2}{\alpha(t)^2}\right) dt < -\frac{1}{4}
\ee
and relatively nonoscillatory with respect to $\tau_0-E$ if
\be\label{eq:relosci}
\sup_{\ell>0} \liminf_{x\to\infty} \frac{1}{\ell} \int_x^{x+\ell}
\left( \mu \frac{u_0(t)^2}{\alpha(t)^2} + \nu \frac{(p_0(t) u_0'(t))^2}{\alpha(t)^2}\right) dt > -\frac{1}{4}.
\ee
\end{corollary}

\noindent
Clearly the precise asymptotic requirement can be removed by a
simple Sturm-type comparison argument (see Lemma~\ref{lem:monowron} below).

In the special case where $p_0$, $q_0$, and $r$ are periodic
functions, one has $\alpha(x)=1$,
$\beta(x)=\pm x$ (with the plus sign if $E$ is a lower band edge and the minus
sign if $E$ is an upper band edge) and can take $\ell$ to be the period. 

Then
$$
C_q = \frac{1}{\ell} \int_x^{x+\ell} u_0(t)^2 dt, \qquad
C_p = \frac{1}{\ell} \int_x^{x+\ell} (p_0(t) u_0'(t))^2 dt,
$$
are constants and (\ref{eq:reloscs}) respectively
(\ref{eq:relosci}) just read
$$
\mu C_q  + \nu C_p  \lessgtr -\frac{1}{4}.
$$
In the special case $p_0=p_1=1$ we recover Rofe-Beketov's well-known result \cite{rb3}--\cite{rb5}
since one can show (see Section~\ref{sec:floq})
$$
C_q =  \frac{|D'(E)|}{\ell^2}
$$
for $r(x)=1$, where $D$ is the Floquet discriminant.
In the special case $\Delta p=0$ we recover the recent extension by Schmidt \cite{kms}.

If $p_0$, $q_0$ are almost periodic and there exists an almost periodic
solution at the band edge $E$, then $E$ is an admissible band edge
($\alpha(x)=1$, $\beta(x)=\pm x$) after Lemma~6.5
in \cite{rbk}. By taking $\ell\to\infty$ in our formulas we recover the
oscillation criteria by Rofe-Beketov (\cite[Thm.~6.12]{rbk}).
In \cite{rbk}, it is furthermore shown that if the spectrum of the operator
$H_0$ has a band-structure, obeying some growth condition, then there
exist almost periodic solutions at the band edge and a formula for
the critical coupling constant in terms of the band edges is provided.

Clearly, as before we can get a whole scale of criteria:

\begin{theorem}\label{thm:mainn}
Fix $n\in\N_0$.
Suppose $E$ is an admissible boundary point of the essential spectrum of $\tau_0$,
with $u_0$, $v_0$ and $\alpha$, $\beta$ as in Definition~\ref{def:regedge}.
Furthermore, suppose that we have $\lim_{x\to\infty}\beta(x)=\infty$ and
\be
\Delta q, \Delta p = O\big(\frac{\beta'}{\alpha^2\beta^2}\big).
\ee
Abbreviate
\be
Q = \frac{1}{\beta'} \left( u_0^2 \Delta q + (p_0 u_0')^2 \Delta p\right).
\ee
Then $\tau_1-E$ is relatively oscillatory with respect to $\tau_0-E$ at $b$ if
\be
\inf_{\ell>0} \limsup_{x\to\infty} \frac{L_n(\beta(x))^2}{\beta(x)^2} \left(\frac{1}{\ell} \int_{x}^{x+\ell}
\beta(t)^2 Q(t) dt - \beta(x)^2 Q_n(\beta(x)) \right) < -\frac{1}{4}
\ee
and relatively nonoscillatory with respect to $\tau_0-E$ at $b$ if
\be
\sup_{\ell>0} \liminf_{x\to\infty} \frac{L_n(\beta(x))^2}{\beta(x)^2} \left(\frac{1}{\ell} \int_{x}^{x+\ell}
\beta(t)^2 Q(t) dt - \beta(x)^2 Q_n(\beta(x)) \right) > -\frac{1}{4}.
\ee
\end{theorem}

As a consequence we get:

\begin{corollary}
Let $\tau_0$ be periodic on $(a,\infty)$ with $r(x)=1$ and let $n\in\N_0$. Define
$$
\mu_c= -\frac{\ell^2}{|D|'(E)},
$$
and suppose
\be
q_1 =  q_0 + \mu_c \Big(Q_n + \frac{\mu}{L_n^2} \Big)+
o\Big(\frac{1}{L_n^2}\Big), \qquad
p_1 = p_0 + o\Big(\frac{1}{L_n^2}\Big).
\ee
Then $\tau_1-E$ is relatively oscillatory with respect to $\tau_0-E$ if
\be
\mu < -\frac{1}{4}
\ee
and relatively nonoscillatory with respect to $\tau_0-E$ if
\be
\mu > -\frac{1}{4}.
\ee
\end{corollary}

\noindent
Again the special case $n=1$ and $\Delta p=0$ is due to \cite{kms}.
The assumption $r(x)=1$ can be dropped, but then $\mu_c$ can no longer
be expressed in terms of the derivative of the Floquet discriminant (alternatively one could also
choose $\alpha(x)=r(x)^{-1/2}$). A non-oscillation result similar in spirit to the Hille-Wintner
result mentioned earlier was given by Khrabustovskii \cite{kh3}.

\section{Relative Oscillation Theory in a Nutshell}
\label{sec:relosc}

The purpose of this section is to provide some further details on relative oscillation
theory and to show how the question of relative (non)oscillation is related to finiteness
of the number of eigenvalues in essential spectral gaps. We refer to
\cite{kt} and \cite{kt2} for further results, proofs, and historical remarks.

Our main object will be the (modified) Wronskian
\be
W_x(u_0,u_1)= u_0(x)\, p_1(x)u_1'(x) - p_0(x)u_0'(x)\, u_1(x)
\ee
of two functions $u_0$, $u_1$ and its zeros. Here we think of $u_0$ and $u_1$ as two
solutions of two different Sturm--Liouville equations $\tau_j u_j = E u_j$ of the type (\ref{stlij}).

Under these assumptions $W_x(u_0,u_1)$ is absolutely continuous and satisfies
\be\label{dwr}
W'_x(u_0, u_1) = (q_1 - q_0) u_0 u_1 +
\left(\frac{1}{p_0} - \frac{1}{p_1}\right) p_0 u_0' p_1 u_1'.
\ee
Next we recall the definition of Pr\"ufer variables $\rho_u$, $\theta_u$
of an absolutely continuous function $u$:
\be\label{eq:defprue}
u(x)=\rho_u(x)\sin(\theta_u(x)), \qquad
p(x) u'(x)=\rho_u(x) \cos(\theta_u(x)).
\ee
If $(u(x), p(x) u'(x))$ is never $(0,0)$ and $u, pu'$ are absolutely continuous, then $\rho_u$ is
positive and $\theta_u$ is uniquely determined once a value of
$\theta_u(x_0)$ is chosen by requiring continuity of $\theta_u$.

Notice that
\be \label{wpruefer}
W_x(u,v)= -\rho_u(x)\rho_v(x)\sin(\Delta_{v,u}(x)), \qquad
\Delta_{v,u}(x) = \theta_v(x)-\theta_u(x).
\ee
Hence the Wronskian vanishes if and only if the two Pr\"ufer angles differ by
a multiple of $\pi$.
We take two solutions $u_j$, $j=0,1$, of $\tau_j u_j =\lam_j u_j$ and
associated Pr\"ufer variables $\rho_j$, $\theta_j$.
We will call the total difference
\be\label{eq:wprueferwron}
\#_{(c,d)}(u_0,u_1) = \ceil{\Delta_{1,0}(d) / \pi} - \floor{\Delta_{1,0}(c) / \pi} -1
\ee
the number of weighted sign flips in $(c,d)$, where we have written $\Delta_{1,0}(x)=
\Delta_{u_1,u_0}$ for brevity.

One can interpret $\#_{(c,d)}(u_0,u_1)$ as the weighted sign flips of the Wronskian
$W_x(u_0,u_1)$, where a sign flip is counted as $+1$ if $q_0-q_1$ and $p_0-p_1$ are positive
in a neighborhood of the sign flip, it is counted as $-1$ if $q_0-q_1$ and $p_0-p_1$ are negative
in a neighborhood of the sign flip. In the case where the differences vanish or are of opposite sign
are more subtle \cite{kt,kt2}.

After these preparations we are now ready for

\begin{definition}\label{def:wsf}
For $\tau_0$, $\tau_1$ possibly singular Sturm--Liouville operators as in (\ref{stlij})
on $(a,b)$, we define
\be
\underline{\#}(u_0,u_1) = \liminf_{d \uparrow b,\,c \downarrow a} \#_{(c,d)}(u_0,u_1)
\quad\mbox{and}\quad
\overline{\#}(u_0,u_1) = \limsup_{d \uparrow b,\,c \downarrow a} \#_{(c,d)}(u_0,u_1),
\ee
where $\tau_j u_j = \lam_j u_j$, $j=0,1$.

We say that $\#(u_0,u_1)$ exists, if
$\overline{\#}(u_0,u_1)=\underline{\#}(u_0,u_1)$, and write
\be
\#(u_0,u_1) = \overline{\#}(u_0,u_1)=\underline{\#}(u_0,u_1).
\ee
in this case.
\end{definition}

\noindent
One can show that $\#(u_0,u_1)$ exists
if $p_0-p_1$ and $q_0-\lam_0r- q_1+\lam_1r$ have the same definite sign
near the endpoints $a$ and $b$.

We recall that in classical oscillation theory
$\tau$ is called oscillatory if a solution of $\tau u = 0$ has infinitely many
zeros.

\begin{definition}\label{def:relosc}
We call $\tau_1$ relatively nonoscillatory with respect to
$\tau_0$, if the quantities $\underline{\#}(u_0, u_1)$ and
$\overline{\#}(u_0, u_1)$ are finite for all solutions
$\tau_j u_j = 0$, $j = 0,1$.
We call $\tau_1$ relatively oscillatory with respect to
$\tau_0$, if one of the quantities $\underline{\#}(u_0, u_1)$ or
$\overline{\#}(u_0, u_1)$ is infinite for some solutions
$\tau_j u_j = 0$, $j = 0,1$.
\end{definition}

\noindent
It turns out that this definition is in fact independent of the solutions
chosen. Moreover, since a Sturm-type comparison theorem holds
for relative oscillation theory, we have

\begin{lemma}\label{lem:monowron}
If $\tau_1$ is relatively oscillatory with respect to $\tau_0$ for $p_1 \leq p_0$, $q_1 \leq q_0$
then the same is true for any $\tau_2$ with $p_2 \leq p_1$, $q_2 \leq q_1$.
Similarly, if $\tau_1$ is relatively nonoscillatory with respect to $\tau_0$ for $p_1 \leq p_0$,
$q_1 \leq q_0$ then the same is true for any $\tau_2$ with $p_1 \leq p_2 \leq p_0$,
$q_1 \leq q_2 \leq q_0$.
\end{lemma}

The connection between this definition and the spectrum is given by:

\begin{theorem} \label{thm:rosc}
Let $H_j$ be self-adjoint operators associated with $\tau_j$, $j=0,1$. Then
\begin{enumerate}
\item
$\tau_0-\lam_0$ is relatively nonoscillatory with respect to $\tau_0-\lam_1$
if and only if $\dim\Ran P_{(\lam_0,\lam_1)}(H_0)<\infty$.
\item
Suppose $\dim\Ran P_{(\lam_0,\lam_1)}(H_0)<\infty$ and
$\tau_1-\lam$ is relatively nonoscillatory with respect to $\tau_0-\lam$ for one
$\lam \in [\lam_0,\lam_1]$.  Then it is relatively nonoscillatory for
all $\lam\in[\lam_0,\lam_1]$ if and only if $\dim\Ran P_{(\lam_0,\lam_1)}(H_1)<\infty$.
\end{enumerate}
\end{theorem}

\noindent
For a practical application of this theorem one needs criteria
when  $\tau_1-\lam$ is relatively nonoscillatory with respect to $\tau_0-\lam$
for $\lam$ inside an essential spectral gap.

\begin{lemma}\label{lem:nonoscingap2}
Let $H_0$ be bounded from below. Suppose $a$ is regular ($b$ singular) and
\begin{enumerate}
\item
$\lim_{x\to b} r(x)^{-1} (q_0(x) - q_1(x)) = 0$, $\frac{q_0}{r}$ is bounded near $b$, and
\item
$\lim_{x\to b} p_1(x) p_0(x)^{-1} = 1$.
\end{enumerate}
Then $\sig_{ess}(H_0)=\sig_{ess}(H_1)$ and
$\tau_1 - \lam$ is relatively nonoscillatory with respect
to $\tau_0 - \lam$ for every $\lam \in \R \backslash\sigma_{ess}(H_0)$.

The analogous result holds for $a$ singular and $b$ regular.
\end{lemma}

\section{Effective Pr\"ufer angles and relative oscillation criteria}
\label{sec:epa}

As in the previous section, we will consider two Sturm--Liouville
operators $\tau_j$, $j=0,1$, and corresponding self-adjoint operators $H_j$, $j=0,1$.
Now we want to answer the question, when a boundary point $E$ of the essential
spectrum of $H_0$ is an accumulation point of eigenvalues of $H_1$.
By Theorem~\ref{thm:rosc} we need to investigate if $\tau_1-E$ is relatively oscillatory
with respect to $\tau_0-E$ or not, that is, if the difference of Pr\"ufer angels
$\Delta_{1,0}=\theta_1-\theta_0$ is bounded or not.

Hence the first step is to derive an ordinary differential equation for $\Delta_{1,0}$. While
this can easily be done, the result turns out to be not very effective for our purpose.
However, since the number of weighted sign flips $\#_{(c,d)}(u_0,u_1)$ is all we
are eventually interested in, any {\em other} Pr\"ufer angle which gives the same
result will be as good:

\begin{definition}
We will call a continuous function $\psi$ a Pr\"ufer angle for the Wronskian $W(u_0,u_1)$, if 
$\#_{(c,d)}(u_0,u_1) = \ceil{\psi(d) / \pi} - \floor{\psi(c) / \pi} -1$ for any $c,d\in(a,b)$.
\end{definition}

Hence we will try to find a more effective Pr\"ufer angle $\psi$ than $\Delta_{1,0}$ for the Wronskian
of two solutions. The right choice was found by Rofe-Beketov \cite{rb2}--\cite{rb5} (see also the recent
monograph \cite{rbk}): 

Let $u_0, v_0$ be two linearly independent solutions of $(\tau_0-\lam) u = 0$ with $W(u_0,v_0)=1$
and let $u_1$ be a solution of $(\tau_1-\lam) u = 0$. Define $\psi$ via
\be\label{def:psi}
W(u_0,u_1)=  -R \sin(\psi), \qquad W(v_0,u_1)= -R \cos(\psi).
\ee
Since $W(u_0,u_1)$ and $W(v_0,u_1)$ cannot vanish simultaneously, $\psi$ is a well-defined
absolutely continuous function, once one value at some point $x_0$ is fixed.

\begin{lemma}
The function $\psi$ defined in (\ref{def:psi}) is a Pr\"ufer angle for the Wronskian $W(u_0,u_1)$.
\end{lemma}

\begin{proof}
Since $W(u_0,u_1)=  -R \sin(\psi) = -\rho_{u_0}\rho_{u_1}\sin(\Delta_{1,0})$ it suffices
to show that $\psi = \Delta_{1,0} \mod 2\pi$ at each zero of the Wronskian. Since we can assume
$\theta_{v_0}-\theta_{u_0}\in(0,\pi)$ (by $W(u_0,v_0)=1$), this follows by
comparing signs of $R\cos(\psi)= \rho_{v_0}\rho_{u_1}\sin(\theta_{u_1}-\theta_{v_0})$.
\end{proof}

\begin{lemma}\label{lem:lemusualprue}
Let $u_0, v_0$ be two linearly independent solutions of $(\tau_0-\lam) u = 0$ with $W(u_0,v_0)=1$
and let $u_1$ be a solution of $(\tau_1-\lam) u = 0$.

Then the Pr\"ufer angle $\psi$ for the Wronskian $W(u_0,u_1)$ defined in (\ref{def:psi})
obeys the differential equation
\be\label{eq:diff1wronski}
\psi' = -\Delta q \big(u_0 \cos (\psi) - v_0 \sin(\psi)\big)^2
- \Delta p \big(p_0 u_0' \cos(\psi) - p_0 v_0' \sin(\psi)\big)^2,
\ee
where
$$
\Delta p = \frac{1}{p_0} - \frac{1}{p_1}, \qquad
\Delta q = q_1 -q_0.
$$
\end{lemma}

\begin{proof}
Observe $R \psi' = -W(u_0,u_1)' \cos(\psi) + W(v_0,u_1)' \sin(\psi)$
and use (\ref{dwr}), (\ref{def:psi}) to evaluate the right hand side.
\end{proof}

\begin{remark}
Special cases of the {\em phase equation} (\ref{eq:diff1wronski}) have been used
in the physics literature before (\cite{ba}, \cite{ca}). Moreover, $\psi$ was originally
not interpreted as Pr\"ufer angle for Wronskians, but defined via
\be
\begin{pmatrix} u_1 \\ p_1 u_1' \end{pmatrix} =
\begin{pmatrix} v_0 & u_0 \\ p_0 v_0'  & p_0 u_0' \end{pmatrix}
\begin{pmatrix} -R \sin(\psi) \\ R \cos(\psi)\end{pmatrix}.
\ee
Augmenting the definition
$$
\begin{pmatrix} u_0 & u_1 \\ p_0 u_0' & p_1 u_1' \end{pmatrix} =
\begin{pmatrix} v_0 & u_0 \\ p_0 v_0'  & p_0 u_0' \end{pmatrix}
\begin{pmatrix} 0 & -R \sin(\psi) \\ 1 & R \cos(\psi)\end{pmatrix},
$$
and taking determinants shows $W(u_0,u_1)= -R \sin(\psi)$. Similarly we obtain
$W(v_0,u_1)= -R \cos(\psi)$ and hence this definition is equivalent to (\ref{def:psi}).
\end{remark}

\noindent
In the case $p_0=p_1$ equation (\ref{eq:diff1wronski}) can be interpreted as the Pr\"ufer
equation of an associated Sturm--Liouville equation with  coefficients given rather implicitly
by means of a Liouville-type transformation of the independent variable. Hence a standard
oscillation criterion of Hille and Wintner \cite[Thm~2.12]{sw} can be used.
This is the original strategy by Rofe-Beketov (see \cite[Sect.~6.3]{rbk}).

In fact, using the transformation $\eta=\tan(\psi)$ it is straightforward to check that $\psi$ satisfies
(\ref{eq:diff1wronski}) if $\eta$ satisfies the Riccati equation
\be
\eta' = -\Delta q \big(u_0 - v_0 \eta\big)^2
- \Delta p \big(p_0 u_0' - p_0 v_0' \eta\big)^2.
\ee
Hence we obtain

\begin{lemma}
Suppose $\Delta p =0$ and $\Delta q>0$. Then $\tau_1$ is relatively (non)oscillatory with
respect to $\tau_0$ if and only if the Sturm--Liouville equation associated with
$$
p^{-1} = \Delta q\, v_0^2 \exp(2 \int \Delta q\, u_0 v_0)>0, \qquad
q =-\Delta q\, u_0^2\, \exp(-2 \int\Delta q\, u_0 v_0)<0
$$
is (non)oscillatory.
\end{lemma}

\begin{proof}
Making another transformation $\phi= \exp(-2 \int \Delta q\, u_0 v_0) \eta$ 
we can eliminate the linear term to obtain the Riccati equation
$$
\phi' = q - \frac{1}{p} \phi^2
$$
for the logarithmic derivative $\phi=\frac{p u'}{u}$ of solutions of the above Sturm--Liouville
equation.
\end{proof}

\noindent
Clearly, an analogous result holds for the case where $\Delta q=0$ and $\Delta p >0$.

Since most oscillation criteria are for the case $p=1$, a Liouville-type transformation
is required before they can be applied. Nevertheless, in order to handle the general case
$\Delta q\ne 0$ and $\Delta p \ne 0$ we will use a more direct approach.

Even though equation (\ref{eq:diff1wronski}) is rather compact, it is still not well suited for
a direct analysis, since in general $u_0$ and $v_0$ will have different growth behaviour (e.g.,
for $\tau_0=-\frac{d^2}{dx^2}$ we have $u_0(x)=1$ and $v_0(x)=x$ at the boundary of the spectrum).
In order to fix this problem Schmidt \cite{kms} proposed to use yet another Pr\"ufer angle $\varphi$
given by the Kepler transformation
\be
\cot(\psi) = \beta_1 \cot(\varphi) + \beta_2,
\ee
where $\beta_1\lessgtr 0$ and $\beta_2$ are arbitrary absolutely continuous functions. It is straightforward to
check that there is a unique choice for $\varphi$ such that it is again absolutely continuous and
satisfies $\floor{\frac{\psi}{\pi}} = \floor{\frac{\varphi}{\pi}}$:
\be\label{def:varphi}
\varphi = \begin{cases}
\sgn(\beta_1) n \pi, & \psi = n\pi,\\
\sgn(\beta_1) n\pi + \arccot(\beta_1^{-1}(\cot(\psi) - \beta_2)), & \psi\in(n\pi,(n+1)\pi),
\end{cases} \qquad n\in\Z,
\ee
where the branch of $\arccot$ is chosen to have values in $(0,\pi)$.
The differential equation for $\varphi$ reads as follows:

\begin{lemma}
Let $u_0, v_0$ be two linearly independent solutions of $(\tau_0-\lam) u = 0$ with $W(u_0,v_0)=1$
and let $u_1$ be a solution of $(\tau_1-\lam) u = 0$.
Moreover, let $\beta_1\lessgtr 0$ and $\beta_2$ be arbitrary absolutely continuous functions.

Then $\sgn(\beta_1) \varphi$, with $\varphi$ defined in (\ref{def:varphi}), is a
Pr\"ufer angle $\varphi$ for the Wronskian $W(u_0,u_1)$ and
obeys  the differential equation
\begin{align} \label{deqvp}
\nn\varphi' =& \frac{\beta_1'}{\beta_1} \sin(\varphi) \cos(\varphi) + \frac{\beta_2'}{\beta_1} \sin^2(\varphi) \\
& - \frac{\Delta q}{\beta_1}\big(\beta_1 u_0 \cos(\varphi) - (v_0 - \beta_2 u_0) \sin(\varphi) \big)^2\\ \nn
& - \frac{\Delta p}{\beta_1}\big(\beta_1 p_0 u_0' \cos(\varphi) - (p_0 v_0' - \beta_2 p_0 u_0')
\sin(\varphi) \big)^2.
\end{align}
\end{lemma}

\begin{proof}
Rewrite (\ref{eq:diff1wronski}) as
$$
\frac{\psi'}{\sin(\psi)^2} =
-\Delta q \big(u_0 \cot(\psi) -  v_0)\big)^2
-\Delta p \big(p_0 u_0' \cot(\psi) -  p_0 v_0')\big)^2.
$$
On the other hand one computes
$$
\frac{\psi'}{\sin(\psi)^2}  = -(\cot(\psi))' = -\left(\beta_1 \cot(\varphi) + \beta_2\right)'
=  \beta_1 \frac{\varphi'}{\sin(\varphi)^2} - \beta_1' \cot(\varphi) - \beta_2'
$$
and solving for $\varphi'$ gives (\ref{deqvp}).
\end{proof}

\noindent
We will mainly be interested in the special case $\beta_1=\beta_2\equiv\beta$, where
\begin{align}\label{eq:prueang2}
\varphi' =& \frac{\beta'}{\beta} \left(\sin^2(\varphi) + \sin(\varphi) \cos(\varphi) \right)\\ \nn
& - \beta\, \Delta q\big(u_0 \cos(\varphi) - \frac{1}{\beta} (v_0 - \beta\, u_0) \sin(\varphi) \big)^2\\ \nn
& - \beta\, \Delta p\big(p_0 u_0' \cos(\varphi) - \frac{1}{\beta} (p_0 v_0' - \beta\, p_0 u_0')
\sin(\varphi) \big)^2.
\end{align}
Note that if $\beta<0$ then not $\varphi$, but $-\varphi$ is a Pr\"ufer angle. However, this choice
will avoid case distinctions later on.

Now we turn to applications of this result. As a warm up we will treat the
case where $E$ is the infimum of the spectrum of $H_0$ and prove Theorem~\ref{thm:gu}.

\begin{proof}[Proof of Theorem~\ref{thm:gu}]
Since $\tau_0-E$ is nonoscillatory, $\tau_1-E$ is relatively (non)oscillatory
with respect to $\tau_0-E$ if and only if $\tau_1-E$ is (non)oscillatory.

Set $\beta= \frac{v_0}{u_0} = \int p_0^{-1} u_0^{-2} dt$ and $\rho=\frac{\beta'}{\beta}=
\frac{1}{p_0 u_0 v_0}$. Now observe that (\ref{eq:prueang2}) reads
\begin{align*}
\varphi' =& \rho \big( \sin^2(\varphi) + \sin(\varphi) \cos(\varphi)
- p_0 v_0^2 u_0^2 \Delta q \cos^2(\varphi)\\ & -
p_0 v_0^2 \Delta p(p_0 u_0' \cos(\varphi) - \frac{1}{v_0} \sin(\varphi))^2  \big)\\
=& \rho \left( \sin^2(\varphi) + \sin(\varphi) \cos(\varphi) 
- p_0 v_0^2 (u_0^2 \Delta q + (p_0 u_0')^2 \Delta p) \cos^2(\varphi) \right)
+o(\rho),
\end{align*}
where we have used (\ref{cond:gu}) in the second step. Now use Corollary~\ref{cor:boundsol}
which is applicable since $\rho>0$ and $\int^b \rho(x) dx = \int^b \frac{\beta'(x) dx}{\beta(x)} =
\lim_{x\to b} \log(\beta(x)) =\infty$.
\end{proof}

\noindent
Now note that Corollary~\ref{cor:khwh} in turn gives us an criterion when the differential equation
for our Pr\"ufer angle has bounded solutions:

\begin{lemma}\label{lem:odebnd}
Fix some $n\in\N_0$, let $Q$ be a locally integrable on $(a,b)$ and suppose $\beta\lessgtr 0$
is absolutely continuous with $\rho=\frac{\beta'}{\beta}>0$ locally bounded and
$\lim_{x\to b} |\beta(x)|=\infty$. Then all solutions of the differential equation
\be
\varphi' = \rho \left(\sin^2(\varphi) + \sin(\varphi) \cos(\varphi) -
\beta^2 Q \cos^2(\varphi) \right) + o\big(\frac{\rho \beta^2}{L_n(\beta)^2}\big)
\ee
tend to $\infty$ if 
$$
\limsup_{x\to b} L_n(\beta(x))^2 \left(Q(x) - Q_n(\beta(x))\right) < -\frac{1}{4}
$$
and are bounded above if 
$$
\liminf_{x\to b} L_n(\beta(x))^2 \left(Q(x) - Q_n(\beta(x))\right) > -\frac{1}{4}.
$$
In the last case all solutions are bounded under the additional assumption
$Q =Q_n(\beta) + O(L_n(\beta)^{-2})$.
\end{lemma}

\begin{proof}
The case $n=0$ is Lemma~\ref{lem:boundsol} and hence we can assume $n\ge 1$.
By a change of coordinates $y=\beta(x)$ we can reduce the claim to the case $\beta(x)=x$
(and $b=\infty$).

Now we start by showing that
\begin{align*}
\varphi' =& \frac{1}{x} \left(\Big(1-\frac{A x^2}{L_n(x)} \Big) \sin^2(\varphi)
+ \sin(\varphi) \cos(\varphi) -
x^2 \Big(Q_n + \frac{B}{4 L_n(x)^2}\Big) \cos^2(\varphi) \right)\\ & + o\Big(\frac{x}{L_n(x)^2}\Big)
\end{align*}
has only bounded solutions if $A+B>-1$ and only unbounded solutions (tending
to $\infty$) if $A+B<-1$. Since the error term $o(x L_n(x)^{-2})$ can be bounded
by $\eps x L_n(x)^{-2}(\sin^2(\varphi)+\cos^2(\varphi))$ it suffices to show this for one equation
in this class by an easy sub/super-solution argument: If $A+B<-1$, then any
solution of one equation with slightly smaller $A$ and $B$ is a sub-solution and hence forces the
solution to go to $\infty$. Similarly, If $A+B>-1$, then any
solution of one equation with slightly smaller $A$ and $B$ is a sub-solution and any 
solution of one equation with slightly larger $A$ and $B$ is a super-solution, which together
bound the solutions.

To see the claim for one equation in this class note that unboundedness (boundedness) of solutions
is equivalent to $\tau_1 = -d^2/dx^2 + Q$ being relatively (non)oscillatory with respect to
$\tau_0 = -d^2/dx^2$. Hence it suffices to choose $\beta_1=x(1+ A x^2 L_n^{-2})$,
$\beta_2= x$ and $Q = Q_n +  (A+B) / (4L_n^2)$ in (\ref{deqvp}) and
invoke Corollary~\ref{cor:khwh}.

Finally, the claim from the lemma follows from this result together with another sub/super-solution
argument. 
\end{proof}

\noindent
The special cases $n=0,1$ are essentially due to Schmidt (\cite[Prop.~3 and~4]{kms}).

With this result, we can now prove Theorem~\ref{thm:gun}:

\begin{proof}[Proof of Theorem~\ref{thm:gun}]
Set $\beta= \frac{v_0}{u_0} = \int p_0^{-1} u_0^{-2} dt$ and
$Q= p_0 u_0^2 (u_0^2 \Delta q + (p_0 u_0')^2 \Delta p)$.
As in the proof of Theorem~\ref{thm:gu}, (\ref{eq:prueang2}) reads
$$
\varphi' =  \rho \left( \sin^2(\varphi) + \sin(\varphi) \cos(\varphi) 
- \beta^2 Q \cos^2(\varphi) \right) +o\big(\frac{\rho\beta^2}{L_n(\beta)^2}\big)
$$
and invoking Lemma~\ref{lem:odebnd} finishes the proof (note that $\psi$ and hence also $\varphi$
is always bounded from below, since $\tau_0$ is nonoscillatory).
\end{proof}

\noindent
One might expect that this theorem remains valid if the conditions are not satisfied pointwise
but in some average sense. This is indeed true and can be shown by taking averages in the
differential equation for the Pr\"ufer angle. Such an averaging procedure was first used by Schmidt
\cite{kms1} and further extended in \cite{kms}.

\begin{theorem}\label{thm:guan}
Suppose $\tau_0-E$ has a positive solution and  let $u_0$ be a minimal positive solution.
Define $v_0$ by d'Alembert's formula (\ref{eq:dAl}) and abbreviate 
\be
Q(x) =  p_0(x) u_0^2(x) \left( u_0(x)^2 \Delta q(x) + (p_0(x) u_0'(x))^2 \Delta p(x)\right), \qquad
\beta(x)= \frac{v_0(x)}{u_0(x)}.
\ee
Suppose
$$
\beta^2 Q =O(1), \qquad
p_0 v_0\, p_0 u_0' \Delta p = o\Big(\frac{\beta^2}{L_n(\beta)}\Big), \quad
p_0 \Delta p = o\Big(\frac{\beta^2}{L_n(\beta)}\Big),
$$
and $\rho= (p_0 u_0 v_0)^{-1}$ satisfies $\rho=o(1)$ and (\ref{condrho}).

Then $\tau_1-E$ is oscillatory if
\be
\inf_{\ell>0} \limsup_{x\to\infty} \frac{L_n(\beta(x))^2}{\beta(x)^2} \left(\frac{1}{\ell} \int_{x}^{x+\ell}
\beta(t)^2 Q(t) dt - \beta(x)^2 Q_n(\beta(x)) \right) < -\frac{1}{4}
\ee
and nonoscillatory if
\be
\sup_{\ell>0} \liminf_{x\to\infty} \frac{L_n(\beta(x))^2}{\beta(x)^2} \left(\frac{1}{\ell} \int_{x}^{x+\ell}
\beta(t)^2 Q(t) dt - \beta(x)^2 Q_n(\beta(x)) \right) > -\frac{1}{4}.
\ee
\end{theorem}

\begin{proof}
Derive the differential equation for $\varphi$ as in the proof of Theorem~\ref{thm:gu} and
then take averages using Corollary~\ref{cor:aver}. Observe that the error
term is preserved by monotonicity of $\frac{\beta^2}{L_n(\beta)^2}$ and (\ref{condrho}).
\end{proof}

\noindent
Now we turn to the case above the infimum of the essential spectrum.

\begin{proof}[Proof of Theorem~\ref{thm:mainn}]
Observe that (\ref{eq:prueang2}) reads
$$
\varphi' = \frac{\beta'}{\beta} \left( \sin^2(\varphi) + \sin(\varphi) \cos(\varphi)
- \beta^2 Q \cos^2(\varphi) \right) +o\Big(\frac{\rho \beta^2}{L_n(\beta)^2}\Big).
$$
Average over a length $\ell$ using Corollary~\ref{cor:aver} and observe that the error
term is preserved by monotonicity of $\frac{\beta^2}{L_n(\beta)^2}$ and (\ref{condrho}).
Now apply Lemma~\ref{lem:odebnd}.
\end{proof}

\begin{corollary}
Suppose
\be
\rho= o(\frac{\beta^2}{L_n(\beta)^2}), \quad\text{and}\quad
\frac{1}{\ell} \int_x^{x+\ell} \frac{u_0(t)^2}{\alpha(t)^2} dt = C_q + o\Big(\frac{\beta^2}{L_n(\beta)^2}\Big)
\ee
for some $\ell>0$. Furthermore, assume
\be
\Delta q =  \frac{\beta'}{\alpha^2 C_q} \Big(Q_n(\beta) + \frac{\mu}{L_n(\beta)^2} \Big)+
o\Big(\frac{\beta'}{\alpha^2 L_n(\beta)^2)}\Big), \qquad
\Delta p = o\Big(\frac{\beta'}{\alpha^2 L_n(\beta)^2}\Big).
\ee
Then $\tau_1-E$ is relatively oscillatory with respect to $\tau_0-E$ if
\be
\mu < -\frac{1}{4}
\ee
and relatively nonoscillatory with respect to $\tau_0-E$ if
\be
\mu > -\frac{1}{4}.
\ee
\end{corollary}

\begin{proof}
It is sufficient to show that
$$
\frac{1}{\ell} \int_x^{x+\ell} \left( \frac{\beta(t)^2}{L_j(\beta(t))^2}
- \frac{\beta(x)^2}{L_j(\beta(x))^2} \right) \frac{u_0(t)^2}{\alpha(t)^2} dt =
o\Big(\frac{\beta^2(x)}{L_n(\beta(x))^2}\Big)
$$
for $j = 0, \dots, n$. Since $u_0 \alpha^{-1}$ is bounded, this follows since
by the mean value theorem and monotonicity of $\beta$ we have
$$
\sup_{t\in[x,x+\ell]} \left|\frac{\beta(t)^2}{L_j(\beta(t))^2} - \frac{\beta(x)^2}{L_j(\beta(x))^2} \right|
\le 2\frac{\beta(x)^2}{L_j(\beta(x))^2} \sum_{k=1}^{j} \frac{\beta(x)}{L_k(\beta(x))}
\sup_{t\in[x,x+\ell]}\rho(t),
$$
finishing the proof (note that $\beta/L_0(\beta)=1$ and $\lim_{\beta\to\infty} \beta/L_k(\beta)=0$
for $k\ge 1$).
\end{proof}

\noindent
Note that the assumptions hold for periodic operators by choosing $\ell$ to be the period.
Furthermore, inspection of the proof shows that if $|\beta|\to\infty$, then
$\rho=o(\beta^2 L_n(\beta)^{-2})$ can be replaced by $\rho=O(\beta^2 L_n(\beta)^{-2})$.

\section{Appendix: Averaging ordinary differential equations}
\label{sec:ode}

In Section~\ref{sec:epa} we have reduced everything to the question if certain
ordinary differential equation have bounded solutions or not. In this section
we collect the required results for these ordinary differential equations. The results
are mainly straightforward generalizations of the corresponding results from \cite{kms}.
All proofs are elementary and we give them for the sake of completeness.

\begin{lemma}\label{lem:boundsol}
Suppose $\rho(x)>0$ (or $\rho(x)<0$) is not integrable near $b$. Then the equation
\be
\varphi'(x) = \rho(x) \bigg(A \sin^2\varphi(x) + \cos\varphi(x)\sin\varphi(x)
+ B \cos^2\varphi(x)\bigg) + o(\rho(x))
\ee
has only unbounded solutions if $4 A B > 1$ and only bounded solutions if $4 A B < 1$.
In the unbounded case we have
\be
\varphi(x) = \left(\frac{\sgn(A)}{2} \sqrt{4 A B-1} + o(1) \right) \int^x \rho(t) dt.
\ee
\end{lemma}

\begin{proof}
By a straightforward computation we have
$$
A \sin^2(\varphi) + \sin(\varphi) \cos(\varphi) + B \cos^2(\varphi) = \frac{A + B}{2}
+ \frac{\sqrt{1 + (A - B)^2}}{2} \cos(2(\varphi-\varphi_0)).
$$
for some constant $\varphi_0= \varphi_0(A,B)$. Hence $\psi(x)=\varphi(x)-\varphi_0$ satisfies
\be\label{eq:psiode}
\psi'(x) = \rho(x) \bigg(\frac{A + B}{2} + \frac{\sqrt{1 + (A - B)^2}}{2} \cos(2\psi(x)) \bigg)
+ o(\rho(x))
\ee

If $4 A B < 1$, we have $|A + B| < \sqrt{1 + (A - B)^2}$ from which it follows that the
right hand side of our differential equation is strictly negative for $\varphi(x) \pmod \pi$ close to $\pi/2$ and strictly positive if $\varphi(x) \pmod \pi$ close to $0$. Hence any solution remains in such
a strip.

If $4 A B > 1$, we have $|A + B| > \sqrt{1 + (A - B)^2}$ and thus the right hand side is always
positive, $\psi'(x) \ge C \rho(x)$, if $A,B>0$ and always negative, $\psi'(x) \le -C \rho(x)$, if $A,B<0$.
Since $\rho$ is not integrable by assumption, $\psi$ is unbounded.

In order to derive the asymptotics, rewrite (\ref{eq:psiode}) as
$$
\psi'(x) = \rho(x)\left(\frac{C + D}{2} \cos^2(\psi(x)) + \frac{C - D}{2} \sin^2(\psi(x))\right) + o(\rho(x)),\\
$$
where $C = A + B$ and  $D = \sqrt{1 + (A-B)^2}$.
Now, introduce
$$
\ti{\psi}(x) = \arctan\left(\sqrt{\frac{C -D}{C + D}} \tan(\psi(x))\right)
$$
and observe $|\psi - \ti{\psi}|<\pi$. Moreover,
$$
\ti{\psi}'(x) = \frac{\rho(x)}{2} \sgn(C + D) \sqrt{C^2 - D^2} + o(\rho(x)).
$$
Hence the claim follows since by assumption $4 A B > 1$, which implies
$\sgn(C + D) = \sgn(A)$.
\end{proof}

\noindent
We will also need the case where $A=1$ and $B$ depends on $x$ but not necessarily converge to a
limit as $x\to b$. However, by a simple sub/super-solution argument we obtain
from our lemma

\begin{corollary}\label{cor:boundsol}
Suppose $\rho(x)>0$ is not integrable near $b$. Then all solutions of the equation
\be
\varphi' = \rho \bigg(\sin^2(\varphi) + \sin(\varphi)\cos(\varphi) - B \cos^2(\varphi)\bigg)
+ o(\rho)
\ee
tend to $\infty$ as $x\to b$ if $B(x)\le B_0$ for some $B_0$ with
$B_0 < -\frac{1}{4}$ and are bounded below if $B(x)\ge B_0$ for some
$B_0$ with $B_0 > -\frac{1}{4}$.
\end{corollary}

\noindent
In addition, we also need to look at averages:
Let $\ell > 0$, and denote by
\be
\ol{g}(x) = \frac{1}{\ell} \int_x^{x+\ell} g(t) dt.
\ee
the average of $g$ over an interval of length $\ell$.

\begin{lemma}\label{lem:aver}
Let $\varphi$ obey the equation
\be
\varphi'(x) = \rho(x) f(x) + o(\rho(x)), \qquad x\in(a,\infty),
\ee
where $f(x)$ is bounded. If
\be\label{condrho}
\frac{1}{\ell} \int_0^\ell \left|\rho(x+t) -\rho(x) \right| dt = o(\rho(x))
\ee
then
\be
\ol{\varphi}'(x) = \rho(x) \ol{f}(x) + o(\rho(x))
\ee

Moreover, suppose $\rho(x)=o(1)$. If $f(x)= A(x) g(\varphi(x))$, where $A(x)$ is bounded
and $g(x)$ is bounded and Lipschitz continuous, then
\be
\ol{f}(x)=  \ol{A}(x) g(\ol{\varphi}) + o(1).
\ee
\end{lemma}

\begin{proof}
To show the first statement observe
\begin{align*}
\ol{\varphi}'(x) &= \frac{\varphi(x+\ell)-\varphi(x)}{\ell} = \frac{1}{\ell} \int_x^{x+\ell}
\rho(t) f(t) dt + o(\rho(x))\\
&= \rho(x) \ol{f}(x) + \frac{1}{\ell} \int_x^{x+\ell}
(\rho(t) -\rho(x)) f(t) dt + o(\rho(x)).
\end{align*}
Now the first claim follows from (\ref{condrho}) since $f$ is bounded. Note that
(\ref{condrho}) implies that the $o(\rho)$ property is preserved under averaging.

To see the second, we use
\begin{align*}
\ol{f}(x) &= \frac{1}{\ell} \int_x^{x+\ell} A(t) g(\varphi(t)) dt\\
&= \ol{A}(x) g(\ol{\varphi}(x)) + \frac{1}{\ell} \int_x^{x+\ell} A(t) (g(\varphi(t)) - g(\ol{\varphi}(x))) dt.
\end{align*}
Since $g$ is Lipschitz we can use the mean value theorem together with
$$
|\varphi(x+t)) - \ol{\varphi}(x)| \le C \sup_{0 \le s \le \ell} \rho(x+s)
$$
to finish the proof.
\end{proof}

\noindent
Condition (\ref{condrho}) is a strong version of saying that $\ol{\rho}(x)=\rho(x)(1+o(1))$ (it is
equivalent to the latter if $\rho$ is monotone). It will be typically fulfilled if $\rho$
decreases (or increases) polynomially (but not exponentially). For example, the condition holds if
$\sup_{t\in[0,1]} \frac{\rho'(x+t)}{\rho(x)} \to 0$.

We have the next result

\begin{corollary}\label{cor:aver}
Let $\varphi$ obey the equation
\be
\varphi' = \rho \bigg(A \sin^2(\varphi) + \sin(\varphi) \cos(\varphi)
+ B \cos^2(\varphi)\bigg) + o(\rho)
\ee
with $A, B$ bounded functions and assume that $\rho=o(1)$ satisfies (\ref{condrho}).
Then the averaged function $\ol{\varphi}$ obeys the equation
\be
\ol{\varphi}' = \rho \bigg(\ol{A} \sin^2(\ol{\varphi}) + \sin(\ol{\varphi}) \cos(\ol{\varphi})
+ \ol{B} \cos^2(\ol{\varphi})\bigg) + o(\rho).
\ee
\end{corollary}

\noindent
Note that in this case $\varphi$ is bounded (above/below) if and only if $\ol{\varphi}$ is bounded
(above/below).
Furthermore, note that if $A(x)$ has a limit, $A(x)=A_0+o(1)$, then $\ol{A}(x)$ can be replaced by
the limit $A_0$.

\section{Appendix: Periodic operators}
\label{sec:floq}

We will now suppose that $r(x)$, $p(x)$, and $q(x)$ are $\ell$-periodic functions. The purpose of
this section is to recall some basic facts from Floquet theory in order to compute the critical
coupling constant for periodic operators in terms of the derivative of the Floquet discriminant.
A classical reference with further details is \cite{ea}.

Denote by $c(z,x)$, $s(z,x)$ a fundamental system of solutions of $\tau u = z u$
corresponding to the initial conditions $c(z,0)=p(0) s'(z,0)=1$,
$s(z,0)= p(0) c'(z,0)=0$. One then calls
\be
M(z) =
\begin{pmatrix} c(z,\ell) & s(z,\ell) \\
p(\ell) c'(z,\ell) & p(\ell) s'(z,\ell) \end{pmatrix}
\ee
the monodromy matrix. Constancy of the Wronskian, $W(c(z),s(z)) = 1$, implies
$\det M(z)=1$ and defining the Floquet discriminant by
$$
D(z) = \tr(M(z)) = c(z,\ell) + p(\ell) s'(z,\ell),
$$
the eigenvalues $\rho_\pm$ of $M$ are called Floquet multipliers,
\be
\rho_\pm(z) =\frac{D(z)  \pm \sqrt{D(z)^2-4}}{2}, \qquad \rho_+(z)\rho_-(z)=1,
\ee
where the branch of the square root is chosen such that $|\rho_+(z)|\le1$.
In particular, there are two solutions
\be
u_\pm(z,x) = c(z,x) + m_\pm(z) s(z,x),
\ee
the Floquet solutions, satisfying
\be
\begin{pmatrix} u_\pm(z,\ell) \\ p(\ell) u_\pm'(z,\ell) \end{pmatrix} =
\rho_\pm(z) \begin{pmatrix} u_\pm(z,0) \\ p(0) u_\pm'(z,0) \end{pmatrix} =
\rho_\pm(z) \begin{pmatrix} 1 \\ m_\pm(z) \end{pmatrix} .
\ee
Here
\be
m_\pm(z) = \frac{\rho_\pm(z) - c(z,\ell)}{s(z,\ell)}
\ee
are called Weyl $m$-functions. The Wronskian of $u_+$ and $u_-$ is given by
\be\label{Wupum}
W(u_-(z),u_+(z))= m_+(z)-m_-(z)= \frac{\sqrt{D(z)^2-4}}{s(z,\ell)}.
\ee
The functions $u_\pm(z,x)$ are exponentially decaying
as $x\to\pm\infty$ if $|\rho_+(z)|<1$, that is, $|D(z)| > 2$, and are bounded if $|\rho_+(z)|=1$,
that is, $|D(z)| \le 2$. Note that $u_+(z)$ and $u_-(z)$ are linearly independent for $|D(z)|\neq 2$.
The spectrum of $H_0$ is purely absolutely continuous and given by
\be
\sig(H_0) = \{ \lam\in\R\,|\, |D(\lam)| \leq 2 \} = \bigcup_{n=0}^\infty [E_{2n},E_{2n+1}].
\ee
It should be noted that $m_\pm(z)$ (and hence also $u_\pm(z,x)$) are meromorphic in
$\C\backslash \sig(H_0)$ with precisely one of them having a simple pole at the zeros
of $s(z,\ell)$ if the zero is in $\R\backslash \sig(H_0)$. If the zero is at a band edge $E_n$ of
the spectrum, both $m_\pm(z)$ will have a square root type singularity.

\begin{lemma}\label{lem:dotD}
For any $z\in\C$ we have
\be
\dot{D}(z)= - s(z,\ell) \int_0^\ell u_+(z,t) u_-(z,t) r(t) dt,
\ee
where the dot denotes a derivative with respect to $z$.
\end{lemma}

\begin{proof}
Let $u(z,x)$, $v(z,x)$ be two solutions of $\tau u = z u$, which are differentiable with respect to $z$,
then integrating (\ref{dwr}) with $u_0=u(z)$ and $u_1=v(z_1)$, dividing by $z_1-z$ and
taking $z_1\to z$ gives
$$
W_\ell(\dot{v}(z),u(z)) - W_0(\dot{v}(z),u(z)) = \int_0^\ell u(z,t) v(z,t) r(t) dt.
$$
Now choose $u(z)=u_-(z)$ and $v(z)=u_+(z)$ and evaluate the Wronskians
\begin{align*}
W_\ell(\dot{u}_+(z),u_-(z)) - W_0(\dot{u}_+(z),u_-(z)) &= 
\dot{\rho}_+(z) \rho_-(z) W(u_+(z),u_-(z))\\
&=
- \frac{\dot{D}(z)}{\sqrt{D(z)^2-4}} W(u_-(z),u_+(z))
\end{align*}
to obtain the formula.
\end{proof}

\noindent
By (\ref{Wupum}) $u_+$ and $u_-$ are linearly independent away from the band
edges $E_n$. At a band edge $E_n$ we have
$u_-(E_n,x)=u_+(E_n,x) \equiv u(E_n,x)$ and a second linearly independent
solution is given by
$$
s(E_n,x), \qquad W(u(E_n), s(E_n)) = 1.
$$
Here we assume without loss of generality that $s(E_n,\ell) \ne 0$ (since we are only interested
in open gaps, this can always be achieved by shifting the base point $x_0=0$ if necessary).
It is easy to check that $s(E_n,x+\ell) = \sig_n s(E_n,x) + s(E_n,\ell) u(E_n,x)$, where
$\sig_n=\rho_\pm(E_n) = \sgn(D(E_n))$. In particular, $s(E_n,x)$ is of the form
$$
s(E_n,x)= \ti{s}(E_n,x) + \frac{\sig_n s(E_n,\ell)}{\ell} x\, u(E_n,x), \qquad
\ti{s}(E_n,x+\ell)= \sig_n \ti{s}(E_n,x)
$$
and thus $u(E_n,x)$, $s(E_n,x)$ satisfy the requirements of Definition~\ref{def:regedge} with
$\alpha(x)=1$ and $\beta(x) = \sgn(D(E_n)) s(E_n,\ell) \ell^{-1} x$. Observe that
$\beta(x)>0$ for an upper band edge $E_{2m}$ and $\beta(x)<0$ for a lower band
edge $E_{2m+1}$. Moreover, note that at the bottom of the spectrum $s(E_0,x)$ is just the
second solution computed from $u(E_0,x)$ by virtue of d'Alembert's formula (\ref{eq:dAl}).
Setting
$$
u_0(x) = \sqrt{\frac{|s(E_n,\ell)|}{\ell}} u(E_n,x), \qquad
v_0(x) = \sqrt{\frac{\ell}{|s(E_n,\ell)|}} s(E_n,x)
$$
we have $\beta(x)=\sgn(D(E_n) s(E_n,\ell)) x$ and
$\ell^{-1}\int_0^\ell u_0(t)^2 r(t) dt = \ell^{-2} |\dot{D}(E_n)|$ by Lemma~\ref{lem:dotD}.

\section*{Acknowledgments}
The authors wish to thank K.M.\ Schmidt and F.S.\ Rofe-Bektov for valuable hints with
respect to literature.


\begin{thebibliography}{xyy}
\bibitem{ba} V.V. Babikov, {\em The Method of Phase Functions in Quantum Mechanics},
3rd ed., Nauka, Moscow, 1988.
\bibitem{ca} F. Calogero, {\em Variable Phase Approach to Potential Scattering},
Academic Press, New York, 1967.
\bibitem{ea} M.S.P. Eastham, {\em The Spectral Theory of Periodic Differential Equations},
Scottish Academic Press, Edinburgh, 1973.
\bibitem{gu} F. Gesztesy and M. \"Unal, {\em Perturbative oscillation criteria
and Hardy-type inequalities}, Math. Nachr.\ {\bf 189}, 121--144 (1998).
\bibitem{hrt1} P. Hartman, {\em On the linear logarithmic-exponential differential
equation of the second-order}, Amer. J. Math. {\bf 70}, 764--779 (1948).
\bibitem{hl1} E. Hille, {\em Nonoscillation theorems}, Trans. Amer. Math. Soc. {\bf 64}, 234--252 (1948).
\bibitem{kh1} V.I. Khrabustovskii, {\em The perturbation of the spectrum of selfadjoint differential
operators with periodic matrix-valued coefficients} (Russian), in Mathematical physics and functional
analysis, No. 4, pp. 117--138. Fiz.-Tekh. Inst. Nizkikh temp. Akad. Nauk Ukr. SSR, 1973.
\bibitem{kh2} V.I. Khrabustovskii, {\em The perturbation of the spectrum of selfadjoint differential
operators of arbitrary order with periodic matrix coefficients} (Russian), in Mathematical physics
and functional analysis, No. V (Russian), pp. 123--140. Fiz.-Tekh. Inst. Nizkikh Temp. 
Akad. Nauk Ukr. SSR, 1974.
\bibitem{kh3} V.I. Khrabustovskii, {\em The discrete spectrum of perturbed differential operators of
arbitrary order with periodic matrix coefficients},Math. Notes {\bf 21}, no. 5--6, 467--472 (1977).
\bibitem{kn} A. Kneser, {\em Untersuchungen \"uber die reellen
Nullstellen der Integrale linearer Differentialgleichungen}, Math. Ann. {\bf 42},
409--435 (1893).
\bibitem{kt} H. Kr\"uger and G. Teschl, {\em Relative oscillation theory, zeros of the Wronskian,
and the spectral shift function},
Preprint: arXiv:math/0703574.
\bibitem{kt2} H. Kr\"uger and G. Teschl, {\em Relative oscillation theory for Sturm--Liouville operators
extended}, J. Funct. Anal. (to appear).
\bibitem{rb1} F.S. Rofe-Beketov, {\em A test for the finiteness of the number
of discrete levels introduced into gaps of a continuous spectrum by
perturbations of a periodic potential}, Soviet Math. Dokl. {\bf 5}, 689--692 (1964).
\bibitem{rb2} F.S. Rofe-Beketov, {\em Spectral analysis of the Hill operator and its perturbations},
Funkcional'ny\"i analiz {\bf 9}, 144--155 (1977) (Russian).
\bibitem{rb3} F.S. Rofe-Beketov, {\em A generalisation of the Pr\"ufer transformation and the discrete
spectrum in gaps of  the continuous one}, Spectral Theory of Operators, 146--153, Baku, Elm, 1979 (Russian).
\bibitem{rb4} F.S. Rofe-Beketov, {\em Spectrum perturbations, the Kneser-type constants and the
effective masses of zones-type potentials}, Constructive Theory of Functions '84, 757--766, Sofia, 1984.
\bibitem{rb5} F.S. Rofe-Beketov, {\em Kneser constants and effective masses for band potentials},
Sov. Phys. Dokl. {\bf 29}, 391--393 (1984).
\bibitem{rbk}
F.S. Rofe-Beketov and A.M. Kholkin, {\em Spectral analysis of differential operators.
Interplay between spectral and oscillatory properties}, World Scientific, Hackensack, 2005.
\bibitem{kms1} K.M. Schmidt, {\em Oscillation of the perturbed Hill equation
and the lower spectrum of radially periodic Schr\"odinger operators
in the plane}, Proc. Amer. Math. Soc. {\bf 127}, 2367--2374 (1999).
\bibitem{kms} K.M. Schmidt, {\em Critical coupling constants and eigenvalue asymptotics of perturbed
periodic Sturm--Liouville operators}, Commun. Math. Phys. {\bf 211}, 465--485 (2000).
\bibitem{sw} C.A. Swanson, {\em Comparison and Oscillation Theory of Linear Differential Equations},
Academic Press, New York, 1968.
\bibitem{wbr} H. Weber, {\em Die Partiellen Differential--Gleichungen der
Mathematischen Physik, Volume 2}, 5th ed., Vieweg, Braunschweig, 1912.
\end{thebibliography}
\end{document}